\documentclass{amsart}
\usepackage{amsmath,amssymb,graphicx,setspace,verbatim,url}
\usepackage{color}
\usepackage{hyperref}
\usepackage{multicol}
\usepackage{enumitem}
\usepackage{float}
\usepackage{soul}
\usepackage{tikz}
\usepackage{listings}
\usepackage{tikz-cd}
\usepackage{fancyvrb}
\usetikzlibrary{snakes,arrows,shapes}
\usetikzlibrary{arrows.meta}

\tikzcdset{
  arrow style={tikz,diagrams={>=Triangle}}
}

\input xy
\xyoption{all}

\newtheorem{prop}{Proposition}[section]

\newtheorem{thm}[prop]{Theorem}
\newtheorem{lemma}[prop]{Lemma}

\newcommand{\Sym}{\mathrm{S}}

\begin{document}
\title{The degrees of the orientation-preserving automorphism groups of toroidal maps and hypermaps}

\author[M. E. Fernandes]{Maria Elisa Fernandes}
\address{Maria Elisa Fernandes,
Center for Research and Development in Mathematics and Applications, Department of Mathematics, University of Aveiro, Portugal
}
\email{maria.elisa@ua.pt}

\author[Claudio Piedade]{Claudio Alexandre Piedade}
\address{Claudio Alexandre Piedade, Centro de Matemática da Universidade do Porto, Universidade do Porto, Portugal
}
\email{claudio.piedade@fc.up.pt}

\date{}
\maketitle

\begin{abstract}
This paper is an exploration  of the faithful transitive permutation representations of the orientation-preserving automorphisms groups of highly symmetric toroidal maps and hypermaps.
The main theorems of this paper give a list of all possible degrees of these specific groups.  This extends prior accomplishments of the authors, wherein their focus was confined to the study of the automorphisms groups of toroidal regular maps and hypermaps. 

In addition the authors bring out the recently developed  {\sc GAP} package {\sc corefreesub} that can be used to find faithful transitive permutation representations of any group. 
With the aid of this powerful tool, the authors  show how Schreier coset graphs of the automorphism groups of toroidal maps and hypermaps can be easily constructed.
\end{abstract}


\noindent \textbf{2010 Math Subj. Class:} 05E18,  20B25, 52B11

\section{Introduction}

A faithful permutation representation of a group refers to a specific type of group action where each group element is represented by a unique permutation of a set $X$, and no two distinct group elements result in the same permutation. In other words, this representation captures all the distinct elements and their actions within the group. The degree of this representation is the size of the set $X$. A faithful permutation representation is valuable because it allows us to understand the structure and behavior of a group by studying how its elements permute the elements of a set. It's called ``faithful'' because it faithfully captures the group's structure without collapsing distinct group elements into the same permutation. This representation helps mathematicians and researchers analyze and classify groups, understand their properties, and explore their relationships with other mathematical objects. The study of minimal faithful degrees and permutation representations is an active area of research in group theory. For some groups, the minimal degree is relatively easy to compute, while for others, it remains an open question or requires sophisticated mathematical techniques. 

The automorphism groups of regular polytopes are  string C-groups, smooth quotients of Coxeter groups with linear diagrams. In particular, these groups are generated by an ordered set of involutions and nonconsecutive involutions of this set commute. Faithful transitive permutation representations of string C-groups are represented by undirected Schreier coset graphs, satisfying some additional properties due the commuting property of the generators \cite{pellicer}. 
These graphs have been an important tool to discover examples of abstract regular polytopes and to accomplish comprehensive classifications of such geometric objects \cite{fl, flm, flm2,Corr,2017CFLM, 2019fl}. 
This inspired the authors to investigate the different ways of representing a group by a graph (corresponding to a faithful transitive permutation representation).
Their research initiated  by the study of the  automorphism groups of toroidal regular maps \cite{FP20Tor, FP21Cor}. Subsequently, they delved into regular hypermaps, followed by locally toroidal regular polytopes \cite{FP21Hyper, FP22Loc}. In all these works, the authors exclusively focused their investigations on regular structures. Now they will expand their focus considering  toroidal chiral maps and hypermaps. The second author with Delgado also constructed a package for {\sc GAP} \cite{GAP}, named ``corefreesub'', to compute faithful transitive permutation representations of groups and their degrees, which is now available online \cite{CFS}.

The automorphism groups of toroidal chiral maps and hypermaps are $2$-generated groups. The two generators are rotations of the map, typically a face-rotation and a vertex-rotation.
The orientation-preserving automorphims groups of toroidal regular maps, which are index two subgroups of the automorphism group of these maps will also be included in our classification.
Similarly to what was done in our previous works we list all possible degrees of the orientation-preserving groups of automorphisms of toroidal maps. 

The correspondence between faithful transitive permutation representations and core-free subgroups, which is significant concept in group theory that relates group actions to subgroup structure, will be central in this work. 
For any faithful transitive permutation representation of a group, the stabilizer subgroup of the corresponding action is core-free. Conversely, for every core-free subgroup $H$ of a group $G$, there exists a faithful and transitive action of $G$ on the set of cosets of $H$. Thus in our classification we give all core-free indexes of the orientation-preserving automorphism groups, also known as rotational group, of toroidal maps and hypermaps.

\section{Toroidal maps and hypermaps}

In this section, we provide a concise overview of toroidal maps and hypermaps, a topic that has been extensively explored by numerous authors\cite{CoxeterGenerators,burnside,CoxeterSelf,CornHypermaps}.

A common approach to creating a toroidal map is to use a rectangular grid that wraps around the torus.
For that reason toroidal maps, which are embeddings of maps on the surface of a torus, are in correspondence with tesselations of the plane.
There are three types of toroidal maps corresponding to the only three regular plane tessellations, whose basic building blocks are one of the following three regular polygons: the square, the triangle or the hexagon.
Let $(0,1)$ and $(1,0)$ be unitary translations of the plane tesselation. Now consider a vector  $(s_1,s_2)$ for some  non-negative integers $s_1$ and $s_2$.
The  toroidal map that is obtained identifying opposite sides of a parallelogram with vertices 
$$(0,0),\,(s_1,s_2), \,(s_1-s_2,s_1+s_2)\mbox{ and }(-s_2,s_1)$$
 for a quadrangular tesselation and 
 $$(0,0),\,(s_1,s_2), \,(-s_2,s_1+s_2) \mbox{ and }(s_1-s_2,s_1+2s_2)$$
 for the a triangular or a hexagonal tesselation.
 The resulting maps are denoted by $\{4,4\}_{(s_1,s_2)}$ if the tiles of the plane tesselation are squares,  $\{3,6\}_{(s_1,s_2)}$ or $\{6,3\}_{(s_1,s_2)}$ when the tiles are, respectively, triangles or hexagons (see the examples of Figure~\ref{wrap31}).
 Similarly a toroidal hypermap, an embedding of a hypergraph on the torus, is obtained from a regular hexagonal tesselation having vertices with two colors (see Figure~\ref{333}). 
 The toroidal hypermap associated with a vector $(s_1,s_2)$ is  denoted by $(3,3,3)_{(s_1,s_2)}$.

   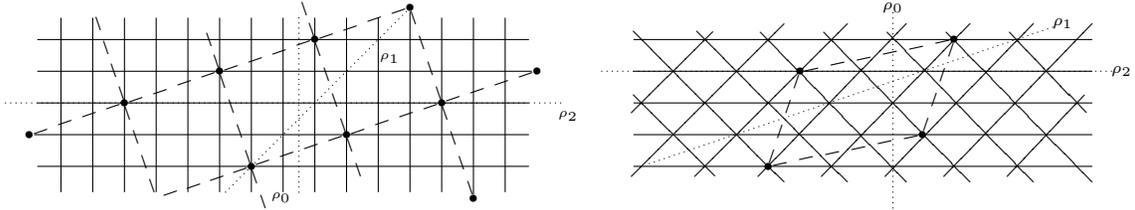
\begin{figure}[h!]
  \begin{tiny}
  \begin{tabular}{cc}
\hspace{-80pt} $\xymatrix@-3pc{
&&&&&&&&&&&&&&&&& &&&&&&&& &&&&&&&&&\\
&&&&\ar@{-}[dddddddddddd]&&*{\bullet}\ar@{-}[dddddddddddd]&&\ar@{-}[dddddddddddd]&&\ar@{-}[dddddddddddd]&&\ar@{-}[dddddddddddd]&&\ar@{-}[dddddddddddd]&&\ar@{-}[dddddddddddd]&\ar@{.}[dddddddddddd]^(0.001){\rho_0}&\ar@{-}[dddddddddddd]&&\ar@{-}[dddddddddddd]&&\ar@{-}[dddddddddddd]&&\ar@{-}[dddddddddddd]&&\ar@{-}[dddddddddddd]&&\ar@{-}[dddddddddddd]&&\ar@{-}[dddddddddddd]&&\ar@{-}[dddddddddddd]&&\\
&&&&&&&&&&&&&&&&& &&&&&&&& &&&&&&&&&\\
&&\ar@{-}[rrrrrrrrrrrrrrrrrrrrrrrrrrrrrrrr]&&&&&&&&&&&&&&& &&&*{\bullet}\ar@{--}[uuul]\ar@{--}[uurrrrrr]&&&&& &&&&&&&&&\\
&&&&&&&&&&&&&&&&& &&&&&&&& &&&&&&&&&\\
&&\ar@{-}[rrrrrrrrrrrrrrrrrrrrrrrrrrrrrrrr]&&&&&&&&&&&&*{\bullet}\ar@{--}[rrrrrruu]\ar@{--}[uuul]&&& &&&&&&&& &&&&&&&&&*{\bullet}\\
&&&&&&&&&&&&&&&&& &&&&&&&& &&&&&&&&&\\
\ar@{.}[rrrrrrrrrrrrrrrrrrrrrrrrrrrrrrrrrrrr]^(0.001){\rho_2}&&\ar@{-}[rrrrrrrrrrrrrrrrrrrrrrrrrrrrrrrr]&&&&&&*{\bullet}\ar@{--}[rrrrrruu]\ar@{--}[uuuuuull]\ar@{--}[ddddddrr]\ar@{--}[lllllldd]&&&&&&&&& &&&&&&&& &&&*{\bullet}\ar@{--}[rrrrrruu]\ar@{--}[uuuuuull]\ar@{--}[ddddddrr]&&&&&&&&&&&&&&\\
&&&&&&&&&&&&&&&&& &&&&&&&& &&&&&&&&&\\
&&*{\bullet}\ar@{-}[rrrrrrrrrrrrrrrrrrrrrrrrrrrrrrrr]&&&&&&&&&&&&&&& &&&&&*{\bullet}\ar@{--}[lluuuuuu]\ar@{--}[rrrrrruu]\ar@{--}[dddr]&&&&&&&&&&&&&&&\\
&&&&&&&&&&&&&&&&& &&&&&&&& &&&&&&&&&\\
&&\ar@{-}[rrrrrrrrrrrrrrrrrrrrrrrrrrrrrrrr]&&&&&&&&&&&&&&*{\bullet}\ar@{--}[uuuuuull]\ar@{--}[rrrrrruu]\ar@{--}[lllllldd]\ar@{--}[rddd]& &&&&&&&& &&&&&&&&&\\
&&&&&&&&&&&&&&&&& &&&&&&&& &&&&&&&&&\\
&&&&&&&&&&*{\bullet}\ar@{.}[urururururururururururur]^(0.2){\rho_1}&&&&&&&&&&&&&&&&&&&&&&&&\\
&&&&&&&&&&&&&&&&&&&&&&&&&&&&&&&&&&&&&&&&&
}$&
\hspace{-10pt}$\xymatrix@-3pc{
&&&&&&&&&&&&&&&&& &&&&&&&& &&&&&&&&&\\ 
&&&&&&&&&&&&&&&&&&&&&&&&& &&&&&&&&&\\
&& *{} \ar@{-}[ddddddddddrrrrrrrrrr] &&&&*{} \ar@{-}[ddddddddddrrrrrrrrrr]&&&&*{} \ar@{-}[ddddddddddrrrrrrrrrr]&&&&*{} \ar@{-}[ddddddddddrrrrrrrrrr]&&&&*{} \ar@{-}[ddddddddddrrrrrrrrrr]&&&&*{} \ar@{-}[ddddddddddrrrrrrrrrr]&&& &\ar@{-}[ddddddrrrrrr]&&&&&& &&\\
&& *{} \ar@{-}[rrrrrrrrrrrrrrrrrrrrrrrrrrrrrr] &&&&&&&&&&&&&&&&&&&&&*{\bullet}\ar@{--}[ddllllllllll]\ar@{--}[ddddddll]&& &&&&&&&&&\\
&& &&&&&&&&&&&&&&&&&&&&&&& &&&&&&&&&\\
&&*{} \ar@{-}[rrrrrrrrrrrrrrrrrrrrrrrrrrrrrr] &&&&&&&&&&&*{\bullet}&&&&&&&&&&&& &&&&&&&&&\\
&&\ar@{-}[ddddddrrrrrr] &&&&&&&&&&&&&&&&&&&&&&& &&&&&&&&&\\
 && *{} \ar@{-}[rrrrrrrrrrrrrrrrrrrrrrrrrrrrrr]&&&&&&&&&&&&&&&&&&&&&&& &&&&&&&&&\\
&&  \ar@{-}[uuuuuurrrrrr]&&&&&&&&&&&&&&&&&&&&&&& &&&&&&&&&\\
*{\rho_2} \ar@{.}[rrrrrrrrrrrrrrrrrrrrrrrrrrrrrrrrrr] && *{} \ar@{-}[rrrrrrrrrrrrrrrrrrrrrrrrrrrrrr] &&& &&&&&&&&&&&&&&&&*{\bullet}&&&& &&&&&&&&&\\
 &&   &&&&&&&&&&&&&&&&&&&&&&& &&&&&&&&&\\
  &&   *{} \ar@{-}[rrrrrrrrrrrrrrrrrrrrrrrrrrrrrr] &&&&&&&&&*{\bullet}\ar@{--}[uuuuuurr]\ar@{--}[uurrrrrrrrrr]&&&&&&&&&&&&&& &&&&&&&&&\\
 &&*{} \ar@{-}[uuuuuuuuuurrrrrrrrrr] &&*{\rho_1} \ar@{.}[uuuuuuuuurrrrrrrrrrrrrrrrrrrrrrrrrrr]&&*{} \ar@{-}[uuuuuuuuuurrrrrrrrrr]&&&&*{} \ar@{-}[uuuuuuuuuurrrrrrrrrr]&&&&*{} \ar@{-}[uuuuuuuuuurrrrrrrrrr]&&&&*{} \ar@{-}[uuuuuuuuuurrrrrrrrrr]&&&&*{} \ar@{-}[uuuuuuuuuurrrrrrrrrr]&& &&\ar@{-}[uuuuuurrrrrr]&&&&&&&&\\
&&   &&&&&&&&&&&&&*{\rho_0} \ar@{.}[uuuuuuuuuuuuu]&& &&&&&&&& &&&&&&&&&\\
 }$
 \end{tabular}
\end{tiny}
\caption{The maps $\{4,4\}_{(3,1)}$ and $\{3,6\}_{(2,1)}$.}
\label{wrap31}
\end{figure}

\begin{figure}[h!]
  \begin{tiny}
 $$\xymatrix@-3pc{
 &&*{\bullet}\ar@{-}[drr]\ar@{-}[dll]&&&&*{\bullet}\ar@{-}[drr]\ar@{-}[dll]&&&&*{\bullet}\ar@{-}[drr]\ar@{-}[dll]&&&&*{\bullet}\ar@{-}[drr]\ar@{-}[dll]&&*{}\ar@{.}[dddddddddddddddddddddd]^(.001){\rho_0}&&*{\bullet}\ar@{-}[drr]\ar@{-}[dll]&&&&*{\bullet}\ar@{-}[drr]\ar@{-}[dll]&&&&*{\bullet}\ar@{-}[drr]\ar@{-}[dll]&&&&*{\bullet}\ar@{-}[drr]\ar@{-}[dll]&&\\
 *{\circ}\ar@{-}[dd]&&&&*{\circ}\ar@{-}[dd]&&&&*{\circ}\ar@{-}[dd]&&&&*{\circ}\ar@{-}[dd]&&&&*{\circ}\ar@{-}[dd]&&&&*{\circ}\ar@{-}[dd]&&&&*{\circ}\ar@{-}[dd]&&&&*{\circ}\ar@{-}[dd]&&&&*{\circ}\ar@{-}[dd]\\
 \\
 *{\bullet}\ar@{-}[drr]&&&&*{\bullet}\ar@{-}[drr]\ar@{-}[dll]&&&&*{\bullet}\ar@{-}[drr]\ar@{-}[dll]&&&&*{\bullet}\ar@{-}[drr]\ar@{-}[dll]&&&&*{\bullet}\ar@{-}[drr]\ar@{-}[dll]&&&&*{\bullet}\ar@{-}[drr]\ar@{-}[dll]&&&&*{\bullet}\ar@{-}[drr]\ar@{-}[dll]&&&&*{\bullet}\ar@{-}[drr]\ar@{-}[dll]&&&&*{\bullet}\ar@{-}[dll]\ar@{.}[ddddddddddddddddllllllllllllllllllllllllllllllll]^(.1){\rho_1}\\
  &&*{\circ}\ar@{-}[dd]&&&&*{\circ}\ar@{-}[dd]&&&&*{\circ}\ar@{-}[dd]&&&&*{\circ}\ar@{-}[dd]&&&&*{\circ}\ar@{-}[dd]&&&&*{\circ}\ar@{-}[dd]&&&&*{\circ}\ar@{-}[dd]&&&&*{\circ}\ar@{-}[dd]\\
 \ar@{.}[rrrrrrrrrrrrrrrrrrrrrrrrrrrrrrrrdddddddddddddddd]^(.01){\rho_2}\\
 &&*{\bullet}\ar@{-}[drr]\ar@{-}[dll]&&&&*{\bullet}\ar@{-}[drr]\ar@{-}[dll]&&&&*{\bullet}\ar@{-}[drr]\ar@{-}[dll]&&&&*{\bullet}\ar@{-}[drr]\ar@{-}[dll]&&&&*{\bullet}\ar@{-}[drr]\ar@{-}[dll]&&&&*{\bullet}\ar@{-}[drr]\ar@{-}[dll]&&&&*{\bullet}\ar@{-}[drr]\ar@{-}[dll]&&&&*{\bullet}\ar@{-}[drr]\ar@{-}[dll]&&\\
 *{\circ}\ar@{-}[dd]&&&&*{\circ}\ar@{-}[dd]&&&&*{\circ}\ar@{-}[dd]&&&&*{\circ}\ar@{-}[dd]&&&&*{\circ}\ar@{-}[dd]&&&&*{\circ}\ar@{-}[dd]&&&&*{\circ}\ar@{-}[dd]&&&&*{\circ}\ar@{-}[dd]&&&&*{\circ}\ar@{-}[dd]\\
 \\
 *{\bullet}\ar@{-}[drr]&&&&*{\bullet}\ar@{-}[drr]\ar@{-}[dll]&&&&*{\bullet}\ar@{-}[drr]\ar@{-}[dll]\ar@{--}[rrrrrrrrrrrrrruuu]&&&&*{\bullet}\ar@{-}[drr]\ar@{-}[dll]&&&&*{\bullet}\ar@{-}[drr]\ar@{-}[dll]&&&&*{\bullet}\ar@{-}[drr]\ar@{-}[dll]&&&&*{\bullet}\ar@{-}[drr]\ar@{-}[dll]&&&&*{\bullet}\ar@{-}[drr]\ar@{-}[dll]&&&&*{\bullet}\ar@{-}[dll]\\
  &&*{\circ}\ar@{-}[dd]&&&&*{\circ}\ar@{-}[dd]&&&&*{\circ}\ar@{-}[dd]&&&&*{\circ}\ar@{-}[dd]&&&&*{\circ}\ar@{-}[dd]&&&&*{\circ}\ar@{-}[dd]&&&&*{\circ}\ar@{-}[dd]&&&&*{\circ}\ar@{-}[dd]\\
 \\
  &&*{\bullet}\ar@{-}[drr]\ar@{-}[dll]&&&&*{\bullet}\ar@{-}[drr]\ar@{-}[dll]&&&&*{\bullet}\ar@{-}[drr]\ar@{-}[dll]&&&&*{\bullet}\ar@{-}[drr]\ar@{-}[dll]&&&&*{\bullet}\ar@{-}[drr]\ar@{-}[dll]&&&&*{\bullet}\ar@{-}[drr]\ar@{-}[dll]&&&&*{\bullet}\ar@{-}[drr]\ar@{-}[dll]&&&&*{\bullet}\ar@{-}[drr]\ar@{-}[dll]&&\\
 *{\circ}\ar@{-}[dd]&&&&*{\circ}\ar@{-}[dd]&&&&*{\circ}\ar@{-}[dd]&&&&*{\circ}\ar@{-}[dd]&&&&*{\circ}\ar@{-}[dd]&&&&*{\circ}\ar@{-}[dd]&&&&*{\circ}\ar@{-}[dd]&&&&*{\circ}\ar@{-}[dd]&&&&*{\circ}\ar@{-}[dd]\\
 \\
 *{\bullet}\ar@{-}[drr]&&&&*{\bullet}\ar@{-}[drr]\ar@{-}[dll]&&&&*{\bullet}\ar@{-}[drr]\ar@{-}[dll]&&&&*{\bullet}\ar@{-}[drr]\ar@{-}[dll]&&&&*{\bullet}\ar@{-}[drr]\ar@{-}[dll]&&&&*{\bullet}\ar@{-}[drr]\ar@{-}[dll]\ar@{--}[rruuuuuuuuu]&&&&*{\bullet}\ar@{-}[drr]\ar@{-}[dll]&&&&*{\bullet}\ar@{-}[drr]\ar@{-}[dll]&&&&*{\bullet}\ar@{-}[dll]\\
  &&*{\circ}\ar@{-}[dd]&&&&*{\circ}\ar@{-}[dd]&&&&*{\circ}\ar@{-}[dd]&&&&*{\circ}\ar@{-}[dd]&&&&*{\circ}\ar@{-}[dd]&&&&*{\circ}\ar@{-}[dd]&&&&*{\circ}\ar@{-}[dd]&&&&*{\circ}\ar@{-}[dd]\\
 \\
  &&*{\bullet}\ar@{-}[drr]\ar@{-}[dll]&&&&*{\bullet}\ar@{-}[drr]\ar@{-}[dll]\ar@{--}[rrrrrrrrrrrrrruuu]\ar@{--}[rruuuuuuuuu]&&&&*{\bullet}\ar@{-}[drr]\ar@{-}[dll]&&&&*{\bullet}\ar@{-}[drr]\ar@{-}[dll]&&&&*{\bullet}\ar@{-}[drr]\ar@{-}[dll]&&&&*{\bullet}\ar@{-}[drr]\ar@{-}[dll]&&&&*{\bullet}\ar@{-}[drr]\ar@{-}[dll]&&&&*{\bullet}\ar@{-}[drr]\ar@{-}[dll]&&\\
 *{\circ}\ar@{-}[dd]&&&&*{\circ}\ar@{-}[dd]&&&&*{\circ}\ar@{-}[dd]&&&&*{\circ}\ar@{-}[dd]&&&&*{\circ}\ar@{-}[dd]&&&&*{\circ}\ar@{-}[dd]&&&&*{\circ}\ar@{-}[dd]&&&&*{\circ}\ar@{-}[dd]&&&&*{\circ}\ar@{-}[dd]\\
 \\
 *{\bullet}\ar@{-}[drr]&&&&*{\bullet}\ar@{-}[drr]\ar@{-}[dll]&&&&*{\bullet}\ar@{-}[drr]\ar@{-}[dll]&&&&*{\bullet}\ar@{-}[drr]\ar@{-}[dll]&&&&*{\bullet}\ar@{-}[drr]\ar@{-}[dll]&&&&*{\bullet}\ar@{-}[drr]\ar@{-}[dll]&&&&*{\bullet}\ar@{-}[drr]\ar@{-}[dll]&&&&*{\bullet}\ar@{-}[drr]\ar@{-}[dll]&&&&*{\bullet}\ar@{-}[dll]\\
  &&*{\circ}&&&&*{\circ}&&&&*{\circ}&&&&*{\circ}&&&&*{\circ}&&&&*{\circ}&&&&*{\circ}&&&&*{\circ}&&&&*{}\\
}$$
\end{tiny}
 \caption{The hypermap $(3,3,3)_{(3,1)}$}\label{333}
\end{figure}
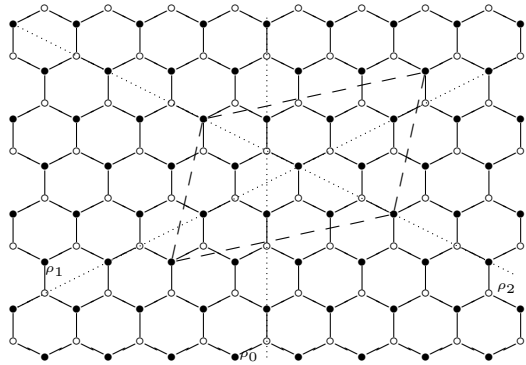

If a toroidal map or hypermap is identical to its mirror image it is \emph{regular}, otherwise it is \emph{chiral}. A toroidal map or hypermap associated with the vector $(s_1,s_2)$ is regular if and only if $s_1s_2(s_1-s_2) = 0$.
 When $(s_1,s_2)\in\{(1,0), (0,1)\}$ we get a degenerated regular tesselation of the torus with either one or two faces. Moreover, with $(s_1,s_2) = (1,1)$ the action on the set of edges is never faithful. Let us consider the cases where the faces of the tesselations of the torus have the shape has the ones of the correspondent planar tesselation. In what follows we assume that $(s_1,s_2)\notin\{(1,0), (0,1), (1,1)\}$.

Plane tesselations are infinite regular polyhedra whose automorphism group is one of the  Coxeter groups $[4,4]$, $[3,6]$ or $[6,3]$ \cite{CoxeterGenerators}.  
The automorphim group of a regular hexagonal tesselation, having vertices with two colors, is also a Coxeter group having a triangular Coxeter diagram.
The groups of automorphims of toroidal maps  and hypermaps are factorizations of these infinite Coxeter groups.

A flag in a map (or hypermap) is a triple of mutual incident  elements (vertex, edge, face) (or (hypervertex, hyperedge, hyperface)).
Flags are \emph{adjacent} if they have exactly  two elements in common.
Consider a flag $(x,y,z)$ (base flag) and their three adjacent flags $(x',y,z)$, $(x,y', z)$ and  $(x,y, z')$. 
Now let $\rho_0$ be the reflexion of the plane tesselation sending $(x,y,z)$ to $(x',y,z)$, $\rho_1$ the reflexion sending $(x,y,z)$ to $(x,y',z)$ and finally,  $\rho_2$ the reflection sending $(x,y,z)$ to $(x,y,z')$.
The group of automorphisms of the plane tesselation is generated by these three involutions. 
Consider the automorphisms $a:=\rho_0\rho_1$, $b:=\rho_1\rho_2$ and $ab:=\rho_0\rho_2$, which are 
rotations, $a$ is a counter-clockwise rotation around a face (or hyperface); $b$ is a counter-clockwise rotation around a vertex (or hypervertex) and $ab$ is a clockwise rotation around an edge (or hyperedge).
 The orders of $a$, $b$ and $ab$ determines the Cotexer group. In the case of the maps the order of $ab$ is $2$ and the orders of $a$ and $b$ are in correspondence with the two parameters of the Coxeter group.
 For the bipartite hexagonal tesselation of the plane the order of $a$, $b$ and $ab$ is $3$.
The rotations $a$ and $b$ are generators of the orientation-preserving automorphism group of the tesselation, commonly called the \emph{rotational group}.
Among these orientation-preserving automorphisms we find unitary translations sending a tile to an adjacent tile.
Let $u$ and $v$ be  unitary translations sending the origin $(0,0)$ to $(1,0)$ and $(0,1)$ respectively.

The rotational group $G$ of a toroidal map or hypermap  is a factorization of the rotational group of the corresponding tesselation by the relation $u^{s_1}v^{s_2}=id$ (where $id$ denotes the identity of $G$).
We consider the translations $u$ and $v$ as defined in Table~\ref{translations}.
As the map $\{6,3\}_{(s_1,s_2)}$ is the dual of $\{3,6\}_{(s_1,s_2)}$, we get the rotational group of  $\{6,3\}_{(s_1,s_2)}$ interchanging the rotations $a$ and $b$. 
Having this in mind the results for the map $\{6,3\}_{(s_1,s_2)}$ can be obtained from the corresponding results for the map $\{3,6\}_{(s_1,s_2)}$.

\begin{table}
\begin{center}
\begin{tabular}{c|c|c}
Map& Translations & $|T:=\langle u,v\rangle|$  \\[5pt]
\hline
$\{4,4\}_{(s_1,s_2)}$ & 
\begin{tabular}{lll}
$u= ab^{-1}$& $u^a=v^{-1}$&  $u^{b}=v^{-1}$ \\
 $v= a^{-1}b$& $v^a=u$&$v^{b}=u$\\
\end{tabular}
&
$s_1^2+s_2^2$ \\
\hline
$\{3,6\}_{(s_1,s_2)}$ 
&
\begin{tabular}{lll}
 $u= ab^{-2}$&$u^{a}=v^{-1}$&$u^{b}=t^{-1}$ \\
$v= a^{-1}b^2$& $v^{a}=t^{-1}$&$ v^{b}=u$\\
$t:=u^{-1}v$&  $t^a = u$& $t^b=v$
\end{tabular}
&
$s_1^2+s_1s_2+s_2^2$\\[5pt]
\hline
$(3,3,3)_{(s_1,s_2)}$ & 
\begin{tabular}{lll}
 $u= ab^{-1}$& $ u^a  = v^{-1}$ & $u^b = v^{-1}$\\
$v= a^{-1}b$ &$v^a= t^{-1}$&$v^b = t^{-1}$\\
$t:=u^{-1}v$& $t^a = u$ & $ t^b = u$  \\
\end{tabular}
&
$s_1^2+s_1s_2+s_2^2$

\end{tabular}
\end{center}
\label{translations}
\end{table}

The subgroup $T$ of $G$ generated by $u$ and $v$ is abelian and is a normal subgroup of $G$. 
Moreover, $T$  acts regularly on the set $V$ of vertices of the toroidal map, hence $|V|=|T|$.
In addition, $G$ acts on the flags with two orbits, hence $|G|=m|V|$ where $m$ is the order of $a$. The translations $u$ and $v$  are conjugate and have order $\frac{|V|}{gcd(s_1,s_2)}$. 

When the map is regular, there exists an automorphism of $G$ sending $a$ to $a^{-1}$ and  $b$ to $b^{-1}$. In this case the group of automorphims of the map is twice bigger then its rotational group.  
In the chiral case the rotational group is precisely the group of automorphisms of the map. 

In what follows $G=\langle a, b\rangle$ is the rotational group of a toroidal map or hypermap and $T=\langle u, v\rangle$ is the group of translations defined in this section.
We now assume that $G$ has a faithful transitive permutation representation of degree $n$. We will determine, for each toroidal map and hypermap, the possible values for the degree $n$ of $G$.
Before we proceed we give some general results that work the same way for any toroidal embedding.

%....................................................................................................................................................................................................................................................................................................................................................................................
%....................................................................................................................................................................................................................................................................................................................................................................................
\section{Preliminary Results}

One consequence of the definition of the translation group $T$ is the following.

\begin{prop}\label{Tsubgroup}
 Any element of the translation subgroup $T$ is of the form $u^i v^j$ with $i\in\{1,\ldots, |u|\}$ and $j\in\{1,\ldots, gcd(s_1,s_2)\}$.
\end{prop}
\begin{proof}
The index of $\langle u\rangle$ in $T$ is equal to $gcd(s_1,s_2)$, thus it is suficient to  prove that $v^{gcd(s_1,s_2)} \in \langle u \rangle$. Let  $x,y\in\mathbb{Z}$ be such that $gcd(s_1,s_2) = xs_1 + ys_2$  (given by Bézout's identity).  

Consider first the toroidal map $\{4,4\}_{(s_1,s_2)}$. Conjugating the equality  $u^{s_1} = v^{-s_2}$ by $a$, we get  $u^{s_2} = v^{s_1}$. 
Hence $ v^{gcd(s_1,s_2)} = v^{xs_1 + ys_2} = u^{-ys_1 + xs_2}\in \langle u\rangle$. 

For the toroidal map $\{3,6\}_{(s_1,s_2)}$ and hypermap $(3,3,3)_{(s_1,s_2)}$, conjugating the equality  $u^{s_1} = v^{-s_2}$ by $a$, we get $v^{s_1}=u^{-(s_1+s_2)}$. Thus $v^{gcd(s_1,s_2)} = u^{-xs_1 - (x+y)s_2}$. Hence, $v^{gcd(s_1,s_2)} \in \langle u \rangle$.
\end{proof}

Now as $T$ is a normal subgroup of $G$, $T$ is fixed-point-free. 
Hence if $T$ is transitive then it acts regularly on $n$. In that case $n=|T|$. 
In what follows we assume that $n\neq |T|$.

\begin{lemma}\label{Gimp}
If $n\neq |T|$ then $G\leq \Sym_k\wr\Sym_m$ where $k$ is the size of a $T$-orbit and $m$ is the number of $T$-orbits. Moreover
$m$ is a divisor of $\frac{|G|}{|T|}$ and  $k=\frac{|T|}{d}$, where $d$ is a divisor of $gcd(s_1,s_2)$.
\end{lemma}
\begin{proof}
Suppose that $n\neq |T|$, then $T$ is intransitive and the $T$-orbits form a block system for $G$. 
Let $m$ be the number of block and $k$ be the size of a block for this block system. 
We have that  $G\leq \Sym_k\wr\Sym_m$. Let us now determine the size of a block.

Consider the induced action of $G$ on the set of $m$ blocks and its induced homomorphim $f:G\to S_m$. 
As $T$ lies in the kernel of this homomorphism, and $Im(f)\cong G/ker(f)$, $|Im(f)|$ is a divisor $\frac{|G|}{|T|}$. Particularly, $m$ is a divisor of $\frac{|G|}{|T|}$.
It remains to prove that $k=|u|d$, where $d$ is a divisor of $gcd(s_1,s_2)$.

Consider the actions $\sigma$ and $\tau$ of  $u$ and $v$, respectively, on a block and let $K:=\langle \sigma,\, \tau\rangle$.
 Let $A:=|\sigma|$, $B:=|K:\langle \sigma\rangle|$ and $C:=|K:\langle \tau\rangle|$.
We have that $K$ has order $AB$ and acts regularly on the block, hence $k=AB$. As $\sigma$ and $\tau$ commute, we have the following
$$K/ \langle \sigma \rangle=\{\langle \sigma \rangle,\, \langle \sigma \rangle\tau, \langle \sigma \rangle\tau^2,\ldots, \langle \sigma \rangle\tau^{B-1}\} \mbox{ and }$$
$$K/ \langle \tau \rangle=\{\langle \tau\rangle,\, \langle \tau \rangle\sigma, \langle \tau \rangle\sigma^2,\ldots, \langle \tau \rangle\sigma^{C-1}\}.$$
Thus $B$ divides $|\tau|$ and $C$ divides $A$. Let $D:=A/C$. As $k=AB=|\tau |C$ we have $|\tau |=DB$.
Now $$|u|=lcm(|\sigma |, |\tau |)=lcm(CD,\,BD)=D\,lcm(C,B)$$ and $$k=AB=DCB=D\,lcm(C,B)\,gcd(C,B)=|u|\,gcd(C,B)=\frac{|T|\,gcd(C,B)}{gcd(s_1,s_2)}$$ 

Let us now prove that $gcd(C,B)$ divides  $gcd(s_1,s_2)$. As both $u^{s_1}$ and $u^{s_2}$ are elements of $\langle v\rangle$, we have that $\sigma^{s_1}$ and $\sigma^{s_2}$ must be elements of $\langle \tau\rangle$, hence $C$ must divide both $s_1$ and $s_2$, meaning it must divide $gcd(s_1,s_2)$. Similarly $\tau^{s_1}$ and $\tau^{s_2}$ are elements of $\langle \sigma\rangle$, and therefore $B$ divides $gcd(s_1,s_2)$. 
Consequently,  $gcd(C,B)$ is a divisor of $gcd(s_1,s_2)$, as wanted.
\end{proof}

\section{Toroidal Maps of type $\{4,4\}$}

In this section let $G$ be the rotational group of  $\{4,4\}_{(s_1,s_2)}.$

\begin{prop}\label{44degdih}
Let  $s_1+s_2>2$.  The subgroups of $G$, $ \langle a\rangle$,   $\langle b\rangle $ and  $\langle ab \rangle$ are core-free.
%If $(s_1,s_2)\in\{(2,0),(0,2)\}$ then $\langle ab \rangle$ is a core-free subgroup of $G$ but  $\langle a\rangle $ and $\langle b\rangle $ are not.
\end{prop}
\begin{proof}
Let $H = \langle a\rangle$ and consider the intersection $H\cap H^b = \langle a \rangle \cap \langle b^{-1} a b\rangle$. 
 If $x\in H\cap H^b$ and $x$ is nontrivial then, for  some $i,j\in\{1,2,3\}$, $ x = a^i = b^{-1} a^j b$, or equivalently $ba^i  = a^jb$. This only can happen if $s_1+s_2\leq 2$ which is not the case.

 For $H=\langle b\rangle$  the intersections $H\cap H^a$ is trivial and for  $H=\langle ab\rangle$ the intersection $H\cap H^b$ is trivial.
 \end{proof}

\begin{prop}\label{44deggcd}
 Let $d$ be a divisor of $gcd(s_1,s_2)$. If $s_1+s_2>2$, then  $\langle u^{s_1/d}v^{s_2/d}\rangle $ and
$\langle a^2,  u^{s_1/d}v^{s_2/d}\rangle$ are core-free subgroups of $G$. Moreover these subgroups of $G$ have indexes $\frac{4(s_1^2+s_2^2)}{d}$ and  $\frac{2(s_1^2+s_2^2)}{d}$, respectively. 
\end{prop}
\begin{proof}
 Let $H = \langle u^{s_1/d}v^{s_2/d}\rangle$, with $d$ being a divisor of $gcd(s_1,s_2)$. Note that $|H|$ is $d$ hence the index $|G:H|$ is as in the statement of this proposition.
 
Consider $\gamma\in H \cap H^a$.  Then  $\gamma = (u^{s_1/d}v^{s_2/d})^i = (v^{-s_1/d}u^{s_2/d})^j$, with $i,j\in\{0,\ldots, d-1\}$. 
This implies that $(u^{s_1}v^{s_2})^{i/d}(u^{-s_2}v^{s_1})^{j/d}=id$. Geometrically this means that the origin $(0,0)$ and the vertex $(x,y)=i/d (s_1,s_2)+j/d(-s_2,s_1)$ are identical. As $i,j\in\{0,\ldots, d-1\}$, this is only possible when $i=j=0$. With this we have shown that $H\cap H^{a}=\{id\}$. 
  
 Now consider $H = \langle a^2, u^{s_1/d}v^{s_2/d}\rangle$, with $d$ being a divisor of $gcd(s_1,s_2)$. 
For $s_1+s_2>2$ we have that $a^2\notin T$ and we have the following equalities, which prove that  $\langle u^{s_1/d}v^{s_2/d}\rangle$ is a normal subgroup of $H$.
\[\begin{array}{l}
 a^{-2} u a^2=a^{-1}v^{-1}a=u^{-1}\\
 a^{-2} v a^2=a^{-1}ua=v^{-1}\\
\end{array}\] 
Hence $H= \langle u^{s_1/d}v^{s_2/d}\rangle \rtimes  \langle a^2\rangle$.

Let $\gamma\in H\cap H^a$. 
$$\gamma = (u^{s_1/d}v^{s_2/d})^i (a^2)^l = (v^{-s_1/d}u^{s_2/d})^j (a^2)^q,$$
with $i,j\in\{0,\ldots,d-1\}$ and $l,q\in\{0,1\}$. Suppose that $(l,q)= (0,0)$. Then $\gamma = (u^{s_1/d}v^{s_2/d})^i = (v^{-s_1/d}u^{s_2/d})^j$, which we gives $\gamma = id$, as we have seen in the previous case. 
If $(l,q)\in\{ (0,1),(1,0)\}$ then $a^2\in T$, a contradiction. Hence $(l,q)= (1,1)$ and, consequently, $(i,j)=(0,0)$, giving that $\gamma \in \langle a^2\rangle$.
This proves that  $H\cap H^a$ is a subgroup of $\langle a^2\rangle$.

Using similar calculations we get that  $H^b \cap H^{ab} \leq  \langle a^2\rangle^b$. 
Hence for  $s_1+s_2>2$,  $ H\cap H^a \cap H^b \cap H^{ab}$ is trivial.
Finally as $|H|=2d$, we have that $|G:H|=2\frac{s_1^2+s_2^2}{d}$.

\end{proof}

\begin{thm}
Let $s_1$ and $s_2$ be nonnegative integers and $D$  the set of divisors of $gcd(s_1,s_2)$. Suppose that $G$ is the rotational group of a toroidal map $\{4,4\}_{(s_1,s_2)}$. The set of all possible degrees of a faithful transitive permutation representation of $G$ is equal to
 $$\Big\{s_1^2+s_2^2\Big\}\cup  \Big\{\frac{2(s_1^2+s_2^2)}{d}\,|\, d\in D\Big\}\cup \Big\{\frac{4(s_1^2+s_2^2)}{d}\,|\, d\in D\Big\}$$
 when  $s_1+s_2>2$  and to $\{8,16\}$ when $(s_1,s_2)\in\{(0,2),(2,0)\}$.
\end{thm}
\begin{proof}
Let $s_1+s_2>2$. By Proposition~\ref{44degdih} $\langle b\rangle$ is core-free subgroup of $G$. As $|G: \langle b\rangle|=s_1^2+s_2^2$ there is a  faithful transitive permutation representation of $G$ of degree $n=s_1^2+s_2^2$.
If $T$ is transitive then the degree of $G$ is equals to the size of $T$, which is $s_1^2+s_2^2$. Then we may assume that $T$ is intransitive. In this case, by Proposition~\ref{Gimp}, the degree of $G$ is among the values given in the statement of this theorem.
Finally, by Proposition~\ref{44deggcd}, there exists a pair of  core-free subgroups of $G$ which have indexes equal to $\frac{2(s_1^2+s_2^2)}{d}$ and $\frac{4(s_1^2+s_2^2)}{d}$.

The cases  $(s_1,s_2)=(0,2)$ and $(s_1,s_2)=(2,0)$ can be  computed using the ``corefreesub'' package \cite{CFS}.

%The case  $(s_1,s_2)\in\{(0,2),(2,0)\}$, $|G|=16$ and $\langle ab\rangle$ is a core-free subgroup of $G$, by Proposition~\ref{44degdih}. In addition there is no transitive group of order $16$ having degree $4$. Consequently $8$ and $16$ are the only possible degrees for $G$ for these particular vectors. This values can be also computed using the coredreesub package  [CITE PACKAGE].
\end{proof}

\section{Toroidal Maps $\{3,6\}$}

In this section let $G$ be the rotational group of  $\{3,6\}_{(s_1,s_2)}.$

%Take into consideration first the rotational group of the toroidal maps $\{3,6\}_{(1,0)}$, $\{3,6\}_{(1,1)}$ and $\{3,6\}_{(2,0)}$. The rotational group of the toroidal maps $\{3,6\}_{(1,0)}$ only has the trivial core-free subgroup, having a faithful transitive permutation representation of degree $6$ (its order). For the rotational group of the toroidal map $\{3,6\}_{(1,1)}$, the trivial subgroup, $\langle ab\rangle$ $\langle a\rangle$ are core-free, giving degrees $18$, $9$ and $6$, respectively, but $\langle b\rangle$ has non-trivial core $\langle b^3 \rangle$. The rotational group of the toroidal map $\{3,6\}$ has a faithful transitive permutation representation with degrees $24$ (trivial subgroup), $12$ ($\langle ab\rangle$ is core-free), $8$ ($\langle a\rangle$ is core-free) and $6$ ($\langle ab, ba\rangle$ is core-free). All these can be found computationally using the corefreesub package for GAP [CITE PACKAGE].

\begin{prop}\label{36degdih}
Let  $s_1+s_2>2$.  The subgroups of $G$, $ \langle a\rangle$,   $\langle b\rangle $ and  $\langle ab \rangle$ are core-free. \end{prop}
\begin{proof}
 Let $H = \langle a\rangle$ and consider the intersection $H\cap H^b = \langle a \rangle \cap \langle b^{-1} a b\rangle$. If $\gamma\in H\cap H^b$ then we have that $ \gamma = a^i = b^{-1} a^j b$, for $i,j\in\{0,1,2\}$. Then we have $ba^i  = a^jb$ which is only possible when flags of adjacent faces are identified, but that is never the case when  $s_1+s_2>2$. Hence $\gamma=id$. 

For $H = \langle b\rangle$ (resp. $H = \langle ab\rangle$) the intersections $H\cap H^a$ (resp. $H\cap H^b$) are trivial.
\end{proof}

\begin{prop}\label{36deggcd}
 Let $d$ be a divisor of $gcd(s_1,s_2)$. If $s_1+s_2>2$, then   $\langle u^{s_1/d}v^{s_2/d}\rangle$ and
$\langle b^3,  u^{s_1/d}v^{s_2/d}\rangle$ are core-free subgroups of $G$. Moreover these subgroups of $G$ have indexes $\frac{6(s_1^2+s_1s_2+s_2^2)}{d}$ and  $\frac{3(s_1^2+s_1s_2+s_2^2)}{d}$, respectively. 
\end{prop}

\begin{proof}
 Let $H = \langle u^{s_1/d}v^{s_2/d}\rangle$, with $d$ being a divisor of $gcd(s_1,s_2)$. 
Consider $\gamma\in H \cap H^a$. 
 Then $\gamma = (u^{s_1/d}v^{s_2/d})^i = (v^{(-s_2-s_1)/d}u^{s_2/d})^j$, with $i,j\in\{0,\ldots, d-1\}$.
Then $u^{\frac{s_1i - s_2j}{d}} v^{\frac{s_1j + s_2i - s_2j}{d}} = id$. Geometrically, this implies that the origin $(0,0)$ and the point with coordinates $(s_1,s_2)i/d+(-s_2,s_1+s_2)j/d$ are vertices
 of the parallelogram used in the construction of the map. As $i,j\in\{0,\ldots, d-1\}$, we must have $i=j=0$. This proves that $H \cap H^a$ is trivial.
  
 Now let $H = \langle b^3,  u^{s_1/d}v^{s_2/d}\rangle$, with $d$ being a divisor of $gcd(s_1,s_2)$. Let us first prove that we can write $H$ as a semi-direct product  $\langle u^{s_1/d}v^{s_2/d}\rangle \rtimes \langle b^3\rangle$. For $s_1+s_2>2$ we have that $b^3 \notin T$ and
 the following equalities show that $ \langle u^{s_1/d}v^{s_2/d}\rangle$ is a normal subgroup of $H$.
 \[\begin{array}{l}
 b^{-3}ub^3=b^{-2}t^{-1}b^2=b^{-1}v^{-1}b=u^{-1}\\[5pt]
 b^{-3}vb^3=b^{-2}ub^2=b^{-1}t^{-1}b=v^{-1}
 \end{array}\]

Let us prove that $H\cap H^{b^2}\leq \langle b^3\rangle$. If $\gamma\in H\cap H^{b^2}$, then $\gamma = (b^3)^l (u^{s_1/d}v^{s_2/d})^i = (b^3)^q (v^{(-s_1-s_2)/d}u^{s_2/d})^j$, with $i,j \in\{0,\ldots,d-1\}$ and $l,q\in\{0,1\}$. 
Now if  $(l,q) = (0,0)$, then, as we have proven before, $(i,j)=(0,0)$, hence $\gamma = id$.  
If $(l,q)\in\{(0,1),(1,0)$ then $b^3\in T$, a contradiction. 
If $(l,q)=(1,1)$, then  $(i,j)=(0,0)$ and $\gamma= b^3$.  Consequently, $H\cap H^{b^2}\leq \langle b^3\rangle$, as claimed.

Similarly we have $H^{a^{-1}} \cap H^{b^2a^{-1}} \leq \langle ab^3a^{-1}\rangle$.
As for $s_1+s_2>2$, $\langle b^3\rangle\cap \langle ab^3a^{-1}\rangle$ is trivial, $H$ is a core-free subgroup of $G$, as wanted.
 
\end{proof}

Combining Lemma~\ref{Gimp} and Proposition~\ref{36degdih}, to determine all the possibilities for the degree $n$ of $G$ it remains to consider the case $m=2$, that is, the case where $T$ has exactly two orbits.
The following proposition shows that in that case $n=2|T|=2(s_1^2+s_1s_2+s_2^2)$.

\begin{prop}\label{36m2}
 If $m=2$ then $k=|T|$.
\end{prop}
\begin{proof}
 Suppose that $m=2$. Let $B_1$ and $B_2$ be the orbits of $T$ and, for $i\in\{1,2\}$ denote by $u_i$ and $v_i$ the actions of $u$ and $v$ on the block $B_i$, respectively. 
 As  $a^3=id$, $a$ must fix the blocks, and by transitivity of $G$, $b$  must swaps the blocks.
 Then $|u_1|=|v_1|$ and $|u_2|=|u_1|$. Hence $|u_1|=|u|$. 
Let $K:=\langle u_1,v_1\rangle$ and $d:=|K:\langle u_1\rangle|=|K:\langle v_1\rangle|$. We have that $d$ is a divisor of $gcd(s_1,s_2)$. 

Let $j\in\{0,\ldots,|u|-1\}$ be such that $u_1^d = v_1^j$. Conjugating this equality  by $a$, $b$ and $ab$, respectively, we get the equalities 
$$v_1^d = u_1^{d-j},\,  v_2^d = u_ 2^{d-j}\mbox{ and }u_2^d = u_2^{d-j}v_2^{j-d}.$$
 From the last two relations we have that $u_2^d = v_2^j$. 
Hence, $u^d = v^j$. 
From the proof of Proposition~\ref{Tsubgroup}, we have that both $d$ and $j$ must be multiples of $gcd(s_1,s_2)$. Since $d$ must divide $gcd(s_1,s_2)$, we get that $d=gcd(s_1,s_2)$. 
As $|u| = \frac{|T|}{gcd(s_1,s_2)}$  then the size of the block is $$k =|K|= |u|d = \frac{|T|}{gcd(s_1,s_2)}\cdot gcd(s_1,s_2) = |T|.$$

\end{proof}

\begin{thm}
Let $s_1$ and $s_2$ be nonnegative integers and $D$  the set of divisors of $gcd(s_1,s_2)$.
 Suppose that $G$ is the rotational group of a toroidal map $\{3,6\}_{(s_1,s_2)}$. 
 The set of all possible degrees of a faithful transitive permutation representation of $G$ is equal to
 \[\begin{array}{c}
 \Big\{s_1^2+s_1s_2+s_2^2,\,2(s_1^2+s_1s_2+s_2^2)\Big\}\cup \Big\{\frac{3(s_1^2+s_1s_2s_2^2)}{d}\,|\, d\in D\Big\}\cup \Big\{\frac{6(s_1^2+s_1s_2s_2^2)}{d}\,|\, d\in D\Big\}\\
 \end{array}\]
 when  $s_1+s_2>2$  and to $\{6,8,12\}$ when $(s_1,s_2)\in\{(0,2),(2,0)\}$.
\end{thm}
\begin{proof} 
Let $s_1+s_2>2$. By Proposition~\ref{36degdih}  $\langle a\rangle$ and $\langle b\rangle$ are core-free subgroup of $G$. As $|G: \langle a\rangle|=2(s_1^2+s_1s_2+s_2^2)$ and $|G: \langle b\rangle|=s_1^2+s_1s_2+s_2^2$ there is a  faithful transitive permutation representation of $G$ on the set of cosets of these two subgroups. If $T$ is transitive then the degree of $G$ is equals to the size of $T$, which is $s_1^2+s_1s_2+s_2^2$. 
Then we may assume that $T$ is intransitive. Hence the remaining degrees given in this theorems are obtained from Propositions~\ref{Gimp},~\ref{36deggcd} and ~\ref{36m2}. 

The cases  $(s_1,s_2)=(0,2)$ and $(s_1,s_2)=(2,0)$ can be  computed using the ``corefreesub'' package \cite{CFS}.
\end{proof}

\section{Toroidal Hypermaps $(3,3,3)$}

In this section let $G$ be the rotational group of  the hypermap $\{3,3,3\}_{(s_1,s_2)}.$

\begin{prop}\label{333deggcd}
 Let $d$ be a divisor of $gcd(s_1,s_2)$. If $s_1+s_2>2$, then   $\langle a\rangle $, $\langle b\rangle $,   $\langle ab \rangle$ and  $\langle u^{s_1/d}v^{s_2/d}\rangle$ are core-free subgroups of $G$. 
 \end{prop}
 \begin{proof}
 The proof is similar to the proof of Propositions~\ref{36degdih} and \ref{36deggcd}.
 \end{proof} 

\begin{thm}

Let $s_1$ and $s_2$ be nonnegative integers with $(s_1,s_2)\notin\{(1,0),(0,1), (1,1)\}$ and $D$  the set of divisors of $gcd(s_1,s_2)$.
 Suppose that $G$ is the rotational group of a toroidal map hypermap $(3,3,3)_{(s_1,s_2)}$. 
 The set of all possible degrees of a faithful transitive permutation representation of $G$ is equal to
 \[\begin{array}{c}
 \Big\{s_1^2+s_1s_2+s_2^2\Big\}\cup \Big\{\frac{3(s_1^2+s_1s_2s_2^2)}{d}\,|\, d\in D\Big\}.\\
 \end{array}\]
\end{thm}

\section{Schreier coset graphs}

Let $G=\langle g_i \,|\,  i\in I\rangle$ be a finite  group.
Suppose that $G$ has a faithful transitive permutation representation of degree $n$ (which corresponds to a core-free subgroup of $G$).
A \emph{Schreier coset graph} of $G$ has $n$ vertices and has a directed edge $(x,y)$ with label $g_i$ whenever $xg_i = y$. When $g_i$ is an involution, the two directed edges $(x,y)$ and $(y,x)$ are replaced by a single undirected edge $\{x,y\}$ with label $g_i$. 
In this section, we  give computational tools to represent  Schreier coset graphs of any group, but as example we consider automorphism groups of toroidal maps and hypermaps.

In \cite{FP20Tor,FP21Cor} the authors gave some examples of Schreier coset  graphs of  toroidal regular maps. 
Due to the complexity of drawing Schreier coset  graph of  toroidal chiral maps and hypermaps by hand, we leveraged the funcionalities offered by the {\sc corefreesub} {\sc GAP} package~\cite{GAP,CFS}.
In what follows, we present a code that can be executed using the {\sc GAP} system, provided that the {\sc corefreesub} package has been installed.
As an example we obtain graphs of minimal degree for the map $\{4,4\}_{(2,1)}$ and  the hypermap $(3,3,3)_{(3,2)}$. The Schreier coset graphs obtained are represented in Figures~\ref{FTPRimage} and \ref{333figTex}.\\

\begin{small}
\begin{verbatim}
 gap> LoadPackage("corefreesub");;
 gap> F := FreeGroup("a","b");;
 gap> s1 := 2 ;; s2 := 1;; 
 gap> G44 := F/[F.1^4, F.2^4, (F.1*F.2)^2, (F.1*F.2^-1)^s1*(F.1^-1*F.2)^s2];;
 gap> FTPRs44 := FaithfulTransitivePermutationRepresentations(G44);
 [ [ a, b ] -> [ (1,2,6,3)(4,10,19,11)(5,13,16,7)(8,18,12,14)(9,15,20,17),
 (1,4,12,5)(2,7,17,8)(3,9,16,10)(6,14,11,15)(13,18,20,19) ], 
 [ a, b ] -> [ (1,2,5,3)(4,8,10,6)(7,9), (1,4)(2,6,9,5)(3,7,10,8) ],
 [ a, b ] -> [ (1,2,4,3), (2,3,5,4) ] ]
 gap> DrawFTPRGraph(FTPRs44[3],rec(layout := "sfdp", gen_name := ["a","b"])); 
\end{verbatim}

\end{small}

\begin{figure}[h!]
 \includegraphics[width =0.7\linewidth]{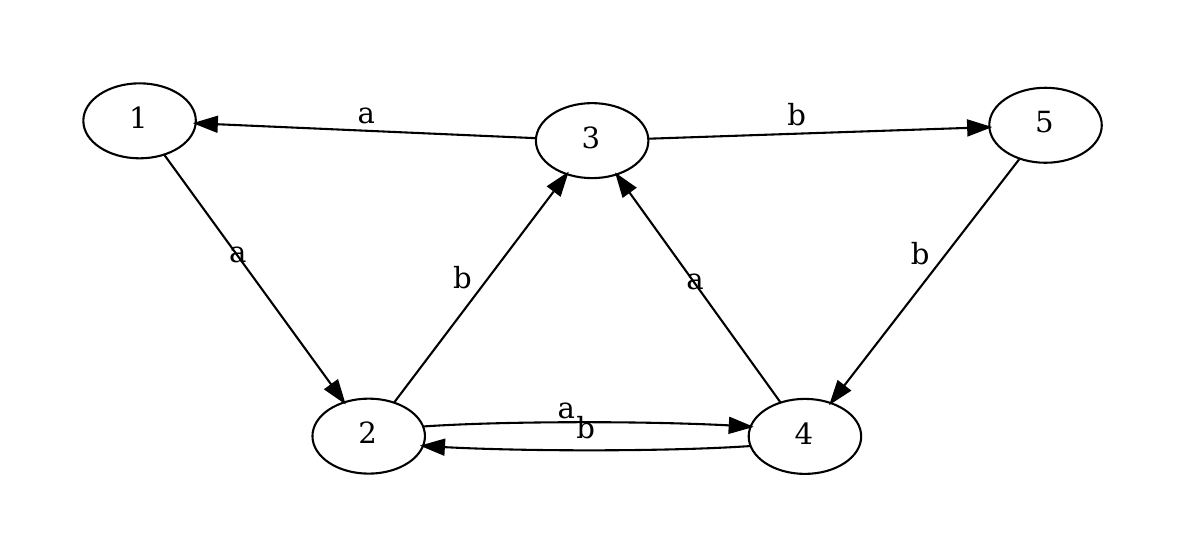}
 \caption{ A Schreier coset graph of $\{4,4\}_{(2,1)}$}
 \label{FTPRimage}
\end{figure}

\begin{small}
\begin{verbatim}
 gap> LoadPackage("corefreesub");; F := FreeGroup("a","b");; s1 := 3 ;; s2 := 2;; 
 gap> G333 := F/[F.1^3, F.2^3, (F.1*F.2)^3, (F.1*F.2^-1)^s1*(F.1^-1*F.2)^s2];;
 gap> FTPRs333 := FaithfulTransitivePermutationRepresentations(G333);
 [ [ a, b ] -> [ (1,2,3)(4,10,11)(5,12,13)(6,14,15)(7,16,17)(8,18,19)(9,20,21)
 (22,37,38)(23,39,40)(24,41,25)(26,42,43)(27,44,45)(28,46,47)(29,48,30)
 (31,49,50)(32,51,52)(33,53,54)(34,55,35)(36,56,57), (1,4,5)(2,6,7)(3,8,9)
 (10,21,22)(11,23,24)(12,25,26)(13,27,14)(15,28,29)(16,30,31)(17,32,18)
 (19,33,34)(20,35,36)(37,57,48)(38,47,39)(40,53,52)(41,51,50)(42,49,56)
 (43,55,44)(45,54,46) ],
 [ a, b ] -> [ (1,2,3)(4,7,8)(5,9,10)(6,11,12)(13,19,17)(14,16,15), (2,4,5)
 (3,6,7)(8,13,14)(9,15,16)(10,17,11)(12,18,19) ] ]
 gap> DrawTeXFTPRGraph(FTPRs333[2],rec(layout := "neato", gen_name := ["a","b"]));
\end{verbatim}
\end{small}

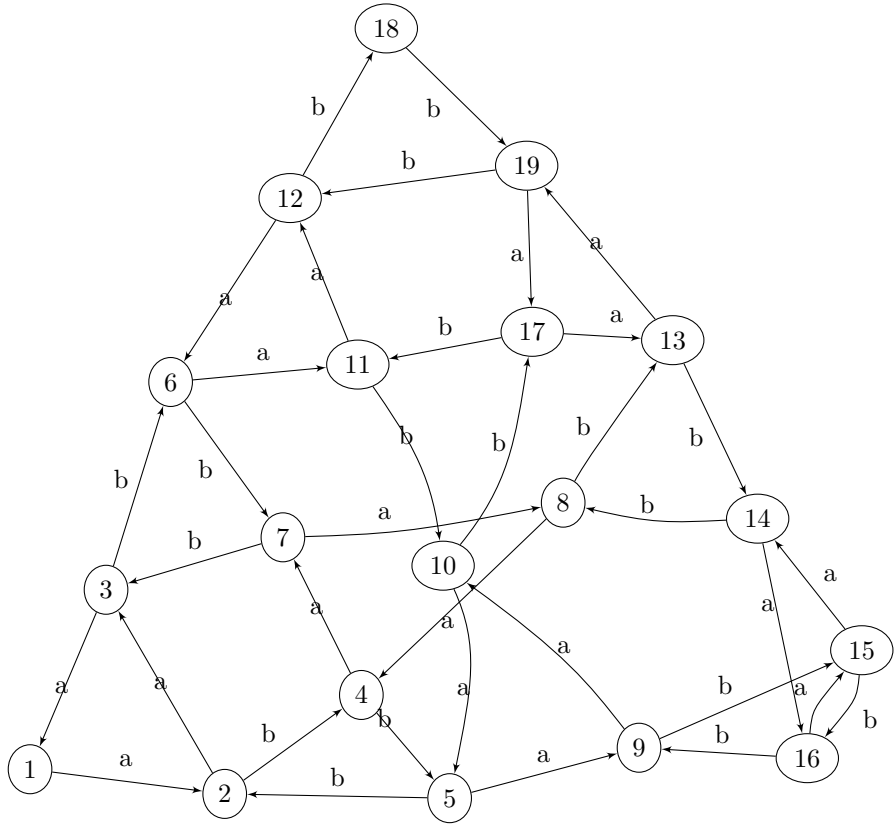
\begin{figure}[h!]
\begin{tikzpicture}[>=latex',line join=bevel,scale=1.00]
\node (1) at (27.0bp,28.912bp) [draw,ellipse] {1};
  \node (2) at (100.36bp,19.613bp) [draw,ellipse] {2};
  \node (3) at (55.637bp,96.602bp) [draw,ellipse] {3};
  \node (4) at (151.85bp,56.924bp) [draw,ellipse] {4};
  \node (6) at (79.974bp,174.96bp) [draw,ellipse] {6};
  \node (7) at (122.29bp,116.38bp) [draw,ellipse] {7};
  \node (5) at (185.12bp,18.0bp) [draw,ellipse] {5};
  \node (8) at (227.91bp,129.45bp) [draw,ellipse] {8};
  \node (9) at (256.49bp,37.026bp) [draw,ellipse] {9};
  \node (10) at (182.66bp,105.68bp) [draw,ellipse] {10};
  \node (15) at (340.44bp,73.797bp) [draw,ellipse] {15};
  \node (11) at (150.52bp,181.53bp) [draw,ellipse] {11};
  \node (12) at (125.02bp,244.33bp) [draw,ellipse] {12};
  \node (13) at (269.2bp,190.73bp) [draw,ellipse] {13};
  \node (17) at (216.24bp,193.88bp) [draw,ellipse] {17};
  \node (18) at (161.26bp,308.3bp) [draw,ellipse] {18};
  \node (19) at (214.15bp,256.53bp) [draw,ellipse] {19};
  \node (14) at (301.2bp,123.42bp) [draw,ellipse] {14};
  \node (16) at (319.91bp,33.087bp) [draw,ellipse] {16};
  \draw [->] (1) ..controls (56.989bp,25.111bp) and (60.179bp,24.706bp)  .. (2);
  \definecolor{strokecol}{rgb}{0.0,0.0,0.0};
  \pgfsetstrokecolor{strokecol}
  \draw (63.135bp,32.402bp) node {a};
  \draw [->] (2) ..controls (84.643bp,46.667bp) and (77.222bp,59.442bp)  .. (3);
  \draw (76.077bp,61.166bp) node {a};
  \draw [->] (2) ..controls (120.73bp,34.375bp) and (122.87bp,35.929bp)  .. (4);
  \draw (116.84bp,42.682bp) node {b};
  \draw [->] (3) ..controls (45.246bp,72.041bp) and (41.687bp,63.63bp)  .. (1);
  \draw (38.808bp,59.959bp) node {a};
  \draw [->] (3) ..controls (64.209bp,124.2bp) and (68.061bp,136.6bp)  .. (6);
  \draw (61.322bp,138.5bp) node {b};
  \draw [->] (4) ..controls (140.87bp,79.01bp) and (138.16bp,84.45bp)  .. (7);
  \draw (134.91bp,89.44bp) node {a};
  \draw [->] (4) ..controls (165.51bp,40.94bp) and (165.64bp,40.787bp)  .. (5);
  \draw (160.58bp,48.363bp) node {b};
  \draw [->] (7) ..controls (95.158bp,108.33bp) and (92.789bp,107.63bp)  .. (3);
  \draw (88.939bp,115.47bp) node {b};
  \draw [->] (7) ..controls (159.13bp,117.57bp) and (170.22bp,117.92bp)  .. (8);
  \draw (160.52bp,125.26bp) node {a};
  \draw [->] (5) ..controls (151.48bp,18.64bp) and (144.61bp,18.771bp)  .. (2);
  \draw (142.87bp,26.209bp) node {b};
  \draw [->] (5) ..controls (214.06bp,25.715bp) and (217.79bp,26.708bp)  .. (9);
  \draw (220.49bp,33.73bp) node {a};
  \draw [->] (9) ..controls (237.1bp,61.759bp) and (228.86bp,72.27bp)  .. (10);
  \draw (228.12bp,74.973bp) node {a};
  \draw [->] (9) ..controls (288.39bp,50.997bp) and (298.81bp,55.563bp)  .. (15);
  \draw (288.96bp,60.936bp) node {b};
  \draw [->] (6) ..controls (95.704bp,153.18bp) and (100.19bp,146.97bp)  .. (7);
  \draw (93.125bp,142.33bp) node {b};
  \draw [->] (6) ..controls (109.17bp,177.68bp) and (111.34bp,177.88bp)  .. (11);
  \draw (114.78bp,185.28bp) node {a};
  \draw [->] (11) ..controls (168.8bp,154.93bp) and (177.87bp,141.74bp)  .. (10);
  \draw (168.76bp,155.22bp) node {b};
  \draw [->] (11) ..controls (141.1bp,204.74bp) and (138.46bp,211.23bp)  .. (12);
  \draw (135.17bp,215.75bp) node {a};
  \draw [->] (8) ..controls (194.23bp,97.038bp) and (184.29bp,87.629bp)  .. (4);
  \draw (184.31bp,84.408bp) node {a};
  \draw [->] (8) ..controls (236.22bp,144.89bp) and (244.25bp,156.06bp)  .. (13);
  \draw (235.6bp,158.48bp) node {b};
  \draw [->] (10) ..controls (196.61bp,75.438bp) and (193.39bp,59.358bp)  .. (5);
  \draw (190.35bp,59.107bp) node {a};
  \draw [->] (10) ..controls (207.35bp,136.09bp) and (209.84bp,152.27bp)  .. (17);
  \draw (203.72bp,152.48bp) node {b};
  \draw [->] (12) ..controls (108.75bp,219.27bp) and (102.17bp,209.14bp)  .. (6);
  \draw (100.67bp,206.26bp) node {a};
  \draw [->] (12) ..controls (138.31bp,267.78bp) and (142.45bp,275.08bp)  .. (18);
  \draw (135.55bp,279.24bp) node {b};
  \draw [->] (13) ..controls (249.24bp,214.58bp) and (241.11bp,224.31bp)  .. (19);
  \draw (240.33bp,227.36bp) node {a};
  \draw [->] (13) ..controls (280.81bp,166.31bp) and (284.79bp,157.94bp)  .. (14);
  \draw (277.97bp,154.25bp) node {b};
  \draw [->] (19) ..controls (179.39bp,251.77bp) and (170.45bp,250.55bp)  .. (12);
  \draw (169.65bp,258.62bp) node {b};
  \draw [->] (19) ..controls (214.92bp,233.39bp) and (215.11bp,227.81bp)  .. (17);
  \draw (210.52bp,222.88bp) node {a};
  \draw [->] (14) ..controls (267.97bp,121.84bp) and (261.57bp,121.53bp)  .. (8);
  \draw (259.61bp,129.18bp) node {b};
  \draw [->] (14) ..controls (307.52bp,92.913bp) and (311.09bp,75.665bp)  .. (16);
  \draw (305.0bp,90.872bp) node {a};
  \draw [->] (16) ..controls (292.92bp,34.764bp) and (292.82bp,34.77bp)  .. (9);
  \draw (287.87bp,42.267bp) node {b};
  \draw [->] (16) ..controls (321.78bp,51.354bp) and (321.84bp,51.512bp)  .. (15);
  \draw (317.31bp,58.933bp) node {a};
  \draw [->] (15) ..controls (325.36bp,92.87bp) and (322.74bp,96.181bp)  .. (14);
  \draw (328.47bp,102.12bp) node {a};
  \draw [->] (15) ..controls (338.58bp,55.531bp) and (338.51bp,55.373bp)  .. (16);
  \draw (343.55bp,47.952bp) node {b};
  \draw [->] (17) ..controls (188.95bp,188.75bp) and (188.01bp,188.57bp)  .. (11);
  \draw (183.47bp,196.16bp) node {b};
  \draw [->] (17) ..controls (243.39bp,192.26bp) and (243.44bp,192.26bp)  .. (13);
  \draw (247.92bp,199.76bp) node {a};
  \draw [->] (18) ..controls (181.23bp,288.75bp) and (186.47bp,283.62bp)  .. (19);
  \draw (179.04bp,278.5bp) node {b};
\end{tikzpicture}
\caption{A Schreier coset graph of $(3,3,3)_{(3,2)}$}
\label{333figTex}
\end{figure}

\section{Acknowledgments}

The author Maria Elisa Fernandes was supported by  the Center for Research and Development 
in Mathematics and Applications (CIDMA) through the Portuguese 
Foundation for Science and Technology 
(FCT - Fundação para a Ciência e a Tecnologia), 
references UIDB/04106/2020 and UIDP/04106/2020. The author Claudio Alexandre Piedade was partially supported by CMUP, member of LASI, which is financed by national funds through FCT – Fundação para a Ciência e a Tecnologia, I.P., under the projects with reference UIDB/00144/2020 and UIDP/00144/2020.

\bibliographystyle{amsplain}
\bibliography{ChiralToroidalMaps}
\end{document}